\documentclass[final,3p,times,authoryear]{elsarticle}

\usepackage{amssymb}
\usepackage{amsthm}
\usepackage{amsmath}
\usepackage{amssymb}
\usepackage{amsfonts}
\usepackage{multirow}
\usepackage{amsthm}
\usepackage{bbm}
\usepackage{xcolor,float}
\usepackage{caption}
\usepackage{subcaption}
\usepackage[hidelinks]{hyperref}

\usepackage{graphicx}
\usepackage{slashbox}
\usepackage{natbib}

\usepackage{todonotes}
\usepackage{comment}
\newtheorem{theorem}{Theorem}
\newtheorem{lemma}{Lemma}
\newtheorem{proposition}{Proposition}
\newtheorem{corollary}{Corollary}

\theoremstyle{definition}

\newtheorem{definition}{Definition}

\newcommand{\R}{\mathbb{R}}

\newcommand{\E}{\mathbb{E}}

\newcommand{\vc}[3]{\overset{#2}{\underset{#3}{#1}}}

\newcommand{\1}{\mathbbm{1}}
\newtheorem{assumption}{Assumption}

\begin{document}
\makeatletter
\def\ps@pprintTitle{%
  \let\@oddhead\@empty
  \let\@evenhead\@empty
  \let\@oddfoot\@empty
  \let\@evenfoot\@oddfoot
}
\makeatother

\begin{frontmatter}



\author[label1]{Tommaso Lando\corref{c1}} 
\fntext[label1]{Department of Economics, University of Bergamo, Italy}
\ead{tommaso.lando@unibg.it}

\author[label1]{Sirio Legramanti} 


\cortext[c1]{Corresponding author}

\title{Bootstrap-based tests for the total time on test and the excess wealth orders}




\begin{abstract}
Given a pair of non-negative random variables $X$ and $Y$, we introduce a class of nonparametric tests for the null hypothesis that $X$ dominates $Y$ in the total time on test order. Critical values are determined using bootstrap-based inference, and the tests are shown to be consistent. The same approach is used to construct tests for the excess wealth order. As a byproduct, we also obtain a class of goodness-of-fit tests for the NBUE (New Better than Used in Expectation) family of distributions.
\end{abstract}



\end{frontmatter}
\section{Introduction}\label{intro}
The \textit{total time on test} (TTT) and the \textit{excess wealth} orders are related and established concepts in the field of stochastic orders \citep{shaked}, with applications in survival analysis, reliability, and economics, among others. However, despite their wide applicability, there is a lack of statistical tests for them. The few exceptions \citep{belzunce2001,belzunce2005} consider the null hypothesis of equality in distribution versus the alternative of strict dominance. In this paper, we follow a complementary approach, which seems to be still unexplored for these orders, and consists in testing the null hypothesis of dominance versus the alternative of non-dominance.

We begin with some notations. Let $H$ be a cumulative distribution function (CDF) defined on the nonnegative half-line, with finite mean $\mu_H$ and quantile function $H^{-1}(p)=\inf\{u:H(u)\geq p\}$. In this paper, we consider the following transforms:
$$T_H(p)=\int_0^{H^{-1}(p)}(1-H(x))dx, \qquad W_H(p)=\int^\infty_{H^{-1}(p)}(1-H(x))dx=\mu_H-T_H(p).$$
$T_H$ is referred to as the TTT transform of $H$, and is a primary tool in reliability and survival analysis \citep{LifeDist}. In particular, when $H$ is the lifetime distribution of a population, $T_H(p)$ gives the expected lifetime of the items in the population {which are smaller than} the $p$-th quantile. In fact, if $Z\sim H$, the TTT transform may be expressed as $T_H(p)=\E(\min(Z,H^{-1}(p)))$. 
 {On the other hand}, $W_H$ is known as the excess wealth transform, the \textit{right spread} function, or \textit{expected shortfall} at level $p$. In particular, {when $H$ is an income distribution}, $W_H(p)$ can be interpreted as the expected additional wealth of the individuals whose wealth exceeds the $p$-th quantile. More formally, $W_H(p)=\E(\max(Z,H^{-1}(p)))$. This transform is commonly employed in insurance mathematics \citep{dhaene}, {but it also has applications in reliability \citep{belzunce2001}. Note that, differently from $T_H,$ the function $W_H$ can be defined even without the nonnegativity assumption. However, here we focus on the nonnegative case to simplify the exposition.}

Given a pair of non-negative random variables, $X$ and $Y$, with CDFs $F$ and $G$, respectively, TTT and excess wealth transforms may be used to define two complementary stochastic orders, recalled in the following definition.
\begin{definition}\label{order}\
\begin{enumerate}\item $X$ dominates $Y$ in the TTT order, abbreviated as $X\geq_{ttt}Y$, if $T_F(p)\geq T_G(p),\forall p\in[0,1]$. 
\item $X$ dominates $Y$ in the excess wealth order, abbreviated as $X\geq_{ew}Y$, if $W_F(p)\geq W_G(p),\forall p\in[0,1].$ 
\end{enumerate}
\end{definition}
The TTT order has been studied, for instance, by \cite{kochar2002}, \cite{li2007}, \cite{belzunce2014} and \cite{belzunce2015}. It compares random variables in terms of size and dispersion. The TTT order is weaker than the usual stochastic order, but it implies the well-known \textit{increasing concave order}, also known as \textit{second-order stochastic dominance} \citep{shaked}. {This means that $X\geq_{ttt}Y$ implies $\E u(X)\geq \E u(Y)$ for every increasing concave function $u$.} One of the main applications of the TTT order is comparing lifetime distributions. In particular, if $X$ and $Y$ represent the random lifetime in two populations, $X\geq_{ttt}Y$ means that, for every proportion $p\in[0,1]$, the expected lifetime of the items in $X$, up to its $p$-th quantile $F^{-1}(p)$, is greater than or equal to the expected lifetime in $Y$, up to the corresponding quantile $G^{-1}(p)$. 

On the other hand, the excess wealth order is {commonly} applied in insurance mathematics, where the main focus is comparing risks \citep{denuit,sordo2009,barmalzan}. {Moreover, this order, also referred to as the \textit{right spread order} \citep{ponce}, also has a natural application in survival analysis and in reliability theory, as it compares expected residual lifetime beyond quantile thresholds; see \citet{shaked1998}, \citet{belzunce2001}, \citet{kochar2007}. For example, consider two systems in which items, with random lifetime $X$ and $Y$, respectively, are tested until the $p100\%$ of them fail. One may be interested in comparing the expected lifetime of the remaining items, which boils down to comparing $W_F(p)$ and $W_G(p)$. In this case, the first system has a larger expected residual lifetime, for every $p$, if $X\geq_{ew}Y$. The properties of the excess wealth order are complementary to those of the TTT order: it is implied by the usual stochastic order and it implies the \textit{increasing convex order} \citep{shaked}. The excess wealth order is also closely related to the \textit{location independent risk order}, introduced by \cite{jewitt} and studied in the economics literature.}

{In spite of the wide literature about these orders, there are few stastical methods available to test whether $X\geq_{ttt}Y$ or $X\geq_{ew}Y$ hold. \cite{belzunce2001} and \cite{belzunce2005} proposed tests for the null hypothesis of equality in distribution versus the alternatives of strict dominance (excluding the case of equal distributions). These tests have the disadvantage that we ignore how they behave when $X\not\geq_{ttt}Y$. Hence, in this paper we follow a complementary approach,
namely, we test $\mathcal{H}^{ttt}_0:X\geq_{ttt}Y$ versus the alternative $\mathcal{H}^{ttt}_1:X\not\geq_{ttt}Y$. In this case,} it is difficult to determine the distribution of the test statistic under $\mathcal{H}_0^{ttt}$, as this contains infinitely many pairs of distributions. Using a similar approach as in \cite{bdlorenz} (see also \cite{bd}), we propose a family of conservative tests for the TTT order, in which critical values are approximated via bootstrap procedures. Given the relation between the TTT and the excess wealth transforms, the same approach may be used to test $\mathcal{H}^{ew}_0:X\geq_{ew}Y$ versus $\mathcal{H}^{ew}_1:X\not\geq_{ew}Y$. Consistency of these tests is established under relatively mild distributional assumptions, which allow for both dependent and independent sampling schemes. 

{Finally, note that the} TTT order can be used to define the \textit{new better than used in expectation} (NBUE) family of distributions, {which} also has applications in survival analysis and reliability. For example, it can be employed in age replacement policies \citep{barlow1975}, applied to shock models \citep{block}, or leveraged to derive bounds for the survival function \citep{cheng,brown}. {The NBUE} condition expresses the notion of positive ageing -- meaning that ageing has an adverse effect on lifetime -- in a weaker sense compared to other popular properties, such as the \textit{increasing hazard rate}, the \textit{increasing hazard rate on average} and the \textit{decreasing mean residual life} \citep{LifeDist}.  In particular, a lifetime random variable $X$ is said to be NBUE if $X/\mu_F\geq_{ttt} E$, where $E$ is a unit exponential random variable. By exploiting this relation, the approach considered in this paper can also be used to define a class of bootstrap-based tests for the null hypothesis $\mathcal{H}^{NBUE}_0:$ ``$X$ is NBUE" versus the alternative $\mathcal{H}^{NBUE}_1:$ ``$X$ is not NBUE". This class of tests may be used as a {complement} to other popular tests (see for instance \cite{hollander,anis2011,anis2014}), which consider the null hypothesis of exponentiality, namely, $\mathcal{H}_0^E:X/\mu_F=_dE$, versus the alternative that $F$ is strictly NBUE, namely, $X/\mu_F>_{ttt} E$. Such tests are not exhaustive, in other words, they are not supposed to work under $\mathcal{H}^{NBUE}_1$. {However, if for a given dataset $\mathcal{H}_0^E$ is rejected whereas $\mathcal{H}^{NBUE}_0$ is not, one may conclude that the underlying distribution is strictly NBUE. Recently, another class of tests for  $\mathcal{H}^{NBUE}_0$ versus $\mathcal{H}^{NBUE}_1$ has been proposed by \cite{landonbue}, with a different approach which does not rely on bootstrap procedures.}

The rest of this paper is organised as follows. In Section~\ref{sect 2}, we introduce a class of test statistics that measure departures from the null hypothesis $\mathcal{H}_0^{ttt}$. In Section~\ref{consistency}, we establish the consistency properties of our class of tests. In Section~\ref{sect nbue}, we obtain tests for the NBUE null hypothesis using the relation between the TTT order and the NBUE class, while in Section~\ref{sect ew} we extend our approach to the excess wealth order. Finite sample properties of our tests are investigated by simulation studies in Section~\ref{sect sim}. {Finally, Section~\ref{sect ex} illustrates the usefulness of the proposed tests through an application.} All proofs are reported in the Appendix.
\section{Construction of the tests}\label{sect 2}
Throughout this paper, let $C[0,1]$ denote the space of bounded continuous functions over $[0,1]$. Given some $v\in C[0,1]$ and $p\geq 1$, $||v||_p$ denotes the $L^p$ norm of $v$, that is, for $p$ finite, $||v||_p=(\int_0^1|v(t)|^pdt)^{1/p},$ while $||v||_\infty=\sup_{[0,1]} |v|$ is the uniform norm. {The symbol $\rightsquigarrow$ denotes weak convergence as in \cite{vw}, while $\to_p$ represents convergence in probability.}

By Definition~\ref{order}, $\mathcal{H}^{ttt}_0$ is false if and only if the difference $\delta(u)=T_G(p)-T_F(p)$ is positive at some point $p\in[0,1]$. Therefore, we can measure the positive part of $\delta$, namely, $\delta_+=\max(0,\delta)$ with the following family of functionals$$\Psi_r(\delta)=||\delta_+||_r, $$
for $r\geq1$, including $r=\infty.$ {In this paper, we focus just on the extreme cases $r=1$ and $r=\infty$. However, the properties of the tests hold for every choice of $r\geq 1$, owing to Proposition 1 in \cite{landonew}.}

Clearly, $\mathcal{H}_0^{ttt}$ is false if and only if $\Psi_r(\delta)>0$. Therefore, we reject the null hypothesis if an empirical version of $\Psi_r(\delta)$, namely $\Psi_r(\widehat{\delta}),$ is significantly large. We will now discuss how to estimate~$\delta$.

\subsection{Estimation}\label{sampling}
Let $\mathcal{X}=\{X_1,...,X_n\}$ be an i.i.d. random sample from $ F $, and $\mathcal{Y}=\{Y_1,...,Y_m\}$ an i.i.d. random sample from~$ G $. In this paper, we deal with two different sampling schemes: independent sampling and matched pairs. Under independent sampling, the samples $\mathcal{X}$ and $\mathcal{Y}$ are independent of each other and the sample sizes $n$ and $m$ may differ. Differently, in the matched-pairs scheme, $n=m$ and we have $n$ i.i.d. pairs $\{(X_1,Y_1),...,(X_n,Y_n)\}$ drawn from a bivariate distribution with marginal CDFs $F$ and $G$, respectively. For our asymptotic analysis, we also assume that, as $ n \to \infty $, also $m\to\infty$, and that $\lim_{n,m\rightarrow \infty} nm/(n+m)=\infty$, and $\lim_{n,m\rightarrow \infty} n/(n+m)=\lambda\in[0,1]$. 

The random samples $\mathcal{X}$ and $\mathcal{Y}$ yield the empirical CDFs $$F_n(x)=\frac1n\sum\nolimits_{i=1}^n\1(X_i\leq x), \qquad G_m(x)=\frac1m\sum\nolimits_{j=1}^m\1(Y_j\leq x).$$ 
We will denote the order statistics of rank $k$ from $\mathcal{X}$ and $\mathcal{Y}$ with $X_{(k)}$ and $Y_{(k)}$, and the sample means by $\overline{X}_n$ and $\overline{Y}_m$, respectively. The empirical {TTT transform of $F$}, denoted by as $T_{F_n}$, is defined as the piecewise linear function which interpolates the jump points of the step function $ \int_0^{F_n^{-1}(p)}(1-F_n(x))dx$. In particular, for $k=0,...,n$, we have that
\begin{equation*}T_{F_n}\left(\frac{k}n\right)=\int_0^{X_{(k)}}(1-F_n(t))dt
=\frac1{n}\sum_{i=1}^{k}\nolimits(n-i+1)(X_{(i)}-X_{(i-1)})=\frac{n-k}{n}X_{(k)}+\frac{1}{n}\sum_{i=1}^{k}\nolimits X_{(i)},\end{equation*}
where $X_{(0)}:=0$. The sample mean is obtained as $T_{F_n}(1)=1/n\sum_{i=1}^{n}X_{(i)}=\overline{X}_n$. Note that, when testing the lifetime of $n$ items, $nT_{F_n}\left(\frac{k}n\right)$ gives the total time on test up to the $k$-th observation. $T_{G_m}$ is defined similarly. Now, an empirical version of $\delta$ is just $$\widehat{\delta}=T_{G_m}-T_{F_n}.$$
It can be shown that, as $n$ and $m$ tend to infinity, $T_{F_n}$ and $T_{G_m}$ converge strongly and uniformly to $T_F$ and $T_G$, respectively, in $[0,1]$ \citep[Theorem 2.1]{barlowzwet}. This implies that $\widehat{\delta}$ also converges strongly and uniformly to $\delta$ in $[0,1]$, under our asymptotic regime.

\section{Consistency}\label{consistency}
\subsection{Asymptotic properties of the test statistic}
 Let us define the empirical TTT processes associated with $F$ and $G$ as $$t^F_n=\sqrt{n}(T_{F_n}-T_F), \qquad t^G_m=\sqrt{m}(T_{G_m}-T_G),$$
respectively. Note that such processes have trajectories in $C[0,1]$. Correspondingly, consider the process $$\Delta_{n,m}=\sqrt{r_{n,m}}\left(\widehat{\delta}-\delta\right)=\sqrt{\tfrac n{n+m}} \ t^G_m-\sqrt{\tfrac m{n+m}} \ t^F_n,$$
where $r_{n,m}=nm/(n+m)$. Recall that, under our assumptions, $r_{nm} \to \infty$ as $n$ and $m$ diverge. To establish the asymptotic properties of our class of tests, we determine the limit behaviour of $\Delta_{n,m}$ via the functional delta method. This requires the derivation of the Hadamard derivative of the map $H\rightarrow T_H$ with respect to the $L^1$ norm, which can be achieved under the following assumption.

\begin{assumption} \label{ass1}
	Both $ F $ and $ G $ are continuously differentiable with strictly positive densities $f$ and $g$, respectively, and have a finite moment of order $2+\epsilon$ for some $\epsilon>0$. Moreover, $F(0)=G(0)=0$.
\end{assumption}

Let $\mathcal{B}$ be a centered Gaussian element of $C[0,1]\times C[0,1]$, with covariance function
$$Cov(\mathcal{B}(x_1,y_1),\mathcal{B}(x_2,y_2))=C(x_1\land x_2,y_1\land y_2)-C(x_1,y_1)C(x_2,y_2).$$
Under the independent-sampling scheme, $C(x_1,y_1)=x_1y_1$ is the product copula, whereas, under the matched-pairs scheme, $C$ is the copula associated with the pair $(X_i,Y_i)$, $i=1,...,n$. Accordingly, $\mathcal{B}_1(x_1)=\mathcal{B}(x_1,1)$ and $\mathcal{B}_2(x_2)=\mathcal{B}(1,x_2)$ are independent Brownian bridge processes under independent-sampling, but they may be dependent under paired sampling.

\begin{theorem}\label{process}
	 Under Assumption~\ref{ass1}, for the sampling schemes discussed in Section~\ref{sampling} we have$$ \sqrt{r_{n,m}}(\widehat{\delta}-\delta)\rightsquigarrow\widetilde{\mathcal{T}}=\sqrt{\lambda}\mathcal{T}_G-\sqrt{1-\lambda}\mathcal{T}_F \quad\text{ in }C[0,1],$$
where
	$$\mathcal{T}_F(p)= \frac{(1-p)\mathcal{B}_1(p)}{f\circ F^{-1}(p)}+\int_0^p \mathcal{B}_1(y)dF^{-1}(y),\quad\mathcal{T}_G(p)= \frac{(1-p)\mathcal{B}_2(p)}{g\circ G^{-1}(p)}+\int_0^p \mathcal{B}_2(y)dG^{-1}(y).$$
	  
\end{theorem}

We measure departures from the null hypothesis with the test statistic $\widehat{\Psi}_r=\sqrt{r_{n,m}}\Psi_r(\widehat{\delta}).$
In particular, ``large" values of $\widehat{\Psi}_r$ lead to rejection of $\mathcal{H}_0^{ttt}$. In fact, under $\mathcal{H}_0^{ttt}:\delta\leq 0,$ $\sqrt{r_{n,m}}\Psi_r(\widehat{\delta})\leq \sqrt{r_{n,m}}\Psi_r(\widehat{\delta}-\delta)$ {by properties of $\Psi_r$. Moreover, it can be shown that $\Psi_r$ satisfies some continuity properties \citep[Proposition~1]{landonew}, hence, the continuous mapping theorem implies that} $\sqrt{r_{n,m}}\Psi_r(\widehat{\delta}-\delta)\rightsquigarrow\Psi_r(\widetilde{\mathcal{T}}).$ In other words, the test statistic $\widehat{\Psi}_r$ is dominated under the null hypothesis by a statistic that has the same asymptotic behaviour as $\Psi_r(\widetilde{\mathcal{T}})$. The same approach was used by \cite{bdlorenz} in testing Lorenz dominance. 

Now, a critical value $c_\alpha$ is determined as the $(1-\alpha)$ quantile of $\Psi_r(\widetilde{\mathcal{T}})$, that is, $P(\Psi_r(\widetilde{\mathcal{T}})> c_\alpha)=\alpha$. For any $\alpha<1/2$, such a quantile is positive, finite, and unique because $\widetilde{\mathcal{T}}$ is a mean-zero Gaussian process. Since, under $\mathcal{H}_1^{ttt}:X\not\geq_{ttt}Y,$ the statistic $\widehat{\Psi}_r/\sqrt{r_{n,m}}=\Psi_r(\widehat{\delta})$ converges in probability to the strictly positive number $\Psi_r(\delta)$, then it is clear that $\widehat{\Psi}_r$ will diverge and exceed the threshold value $c_\alpha$. These properties are summarised in the following lemma. The proof is omitted as it can be obtained using similar arguments as in Lemma 4 of \cite{bdlorenz}.
\begin{lemma}\label{lemmatest}
	\begin{enumerate}\
		\item Under $\mathcal{H}^{ttt}_0$, $\sqrt{r_{n,m}}\Psi_r(\widehat{\delta})\leq \sqrt{r_{n,m}}\Psi_r(\widehat{\delta}-\delta)\rightsquigarrow\Psi_r(\widetilde{\mathcal{T}})$. Moreover, for any $\alpha<1/2$, the $(1-\alpha)$ quantile of the distribution of $\Psi_r(\widetilde{\mathcal{T}})$ is positive, finite, and unique.
		\item Under $\mathcal{H}_1^{ttt}$, $\sqrt{r_{n,m}}\widehat{\Psi}_r\rightarrow_p \infty$. 
	\end{enumerate}
\end{lemma}
Now, critical values may be approximated by simulating the distribution of $\Psi_r(\widetilde{\mathcal{T}})$ via bootstrap procedures.

\subsection{Bootstrap-based decision rule}\label{boot}
Let us denote the bootstrap estimators of the empirical CDFs $F_n$ and $G_m$, respectively, as $F^*_n$ and $G^*_m$. These may be written as
$$F^*_n(x)=\frac1n\sum\nolimits_{i=1}^n M_i^{(1)}\1(x\leq X_i), \qquad G^*_m(x)=\frac1m\sum\nolimits_{i=1}^m M_i^{(2)}\1(x\leq Y_i),$$
where $M^{(1)}=(M_1^{(1)},...,M^{(1)}_n)$ and $M^{(2)}=(M_1^{(2)},...,M^{(2)}_m)$ are independent of the data and are drawn from a multinomial distribution according to the chosen sampling scheme. In particular, under the independent-sampling scheme, $M^{(1)}$ and $M^{(2)}$ are independently drawn from multinomial distributions with uniform probabilities over $n$ and $m$ trials, respectively. Under the matched-pairs scheme, we have $M^{(1)}=M^{(2)}$, drawn from the multinomial distribution with uniform probabilities over $n=m$ trials, which means that we sample (with replacement) pairs of data, from the $n$ pairs $\{(X_1,Y_1),...,(X_n,Y_n)\}$. 
Correspondingly, the bootstrap estimators of $T_{F_n}$ and $T_{G_m}$ are given, respectively, by $T_{F^*_n}$ and $T_{G^*_m}$. Accordingly, we have 
$\widehat{\delta}^*=T_{G^*_n}-T_{F^*_n}.$
By applying the functional delta method for the bootstrap (see Section 3.9.3 of \citet{vw}), one may show that $\sqrt{r_{n,m}}\Psi_r( \widehat{\delta}^*-\widehat{\delta})$ has the same limiting distribution as~$\sqrt{r_{n,m}}\Psi_r(\widetilde{\mathcal{T}}) $. This allows the approximation of the $p$-values $p=P\{\sqrt{r_{n,m}}\Psi_r( \widehat{\delta}^*-\widehat{\delta})>\sqrt{r_{n,m}}\Psi_r( \widehat{\delta})\},$ using $K$ bootstrap replications:
$$p\approx \frac1K \sum\nolimits_{k=1}^K \1(\sqrt{r_{n,m}}\Psi_r( \widehat{\delta}_k^*-\widehat{\delta})>\sqrt{r_{n,m}}\Psi_r( \widehat{\delta})),$$
where $\widehat{\delta}_k^*$ is the $k$-th resampled realisation of $\widehat{\delta}^*$. We reject $\mathcal{H}_0^{ttt}$ whenever $p<\alpha$. 

As proved in the following proposition, the bootstrap approximation of the critical values is consistent. This allows us to exploit the results of Lemma~\ref{lemmatest} and establish the asymptotic behaviour of our class of tests.

\begin{proposition}
	\label{ptest}
	Under Assumption~\ref{ass1} and the sampling schemes in Section~2.1,
	\begin{enumerate}
		\item If $\mathcal{H}^{ttt}_0$ is true, $\lim_{n\rightarrow\infty}P\{\text{reject }\mathcal{H}^{ttt}_0\}\leq\alpha$; 
		\item If $\mathcal{H}^{ttt}_0$ is false, $\lim_{n\rightarrow\infty}P\{\text{reject }\mathcal{H}^{ttt}_0\}=1.$\end{enumerate}
\end{proposition}
{The results obtained in this section may be extended to the case of randomly right-censored data, which is common in life testing problems. In particular, the proposed approach still works if we replace the empirical CDF with the Kaplan-Meier estimator \citep{km}, tailored to deal with this kind of data. The weak convergence of the Kaplan-Meier process has been derived by \cite{gill1983}. More recently, \cite{dobler2019} proved that the bootstrap version of such a process converges weakly to the same limit. To establish the asymptotic properties of the tests in this case, we need to assume that the CDFs under analysis have compact supports, with strictly positive and continuous densities. This allows us to leverage Lemma 3.10.24 in \cite{vw}, obtain the Hadamard differentiability of the inverse map with respect to the uniform norm, and derive that of the TTT transform as in Theorem~\ref{process}. Finally, the properties of the tests follow again by the functional delta method.}

\section{Goodness-of-fit tests for the NBUE family}\label{sect nbue}
The classic way of expressing the NBUE condition is requiring that, for every $t>0$, $\E(X-t|X>t)\leq \E(X) .$ This means that an item of any age $t$ has a mean residual life smaller than or equal to the expected lifetime of a new item in the same population. Generally, ageing properties can also be defined in terms of a suitable stochastic order between the random variable of interest $X$ and the unit exponential $E$, which indeed expresses the concept of ``no ageing" thanks to its memoryless property. In this regard, as discussed in Section \ref{intro}, $X$ is NBUE if $X/\mu_F\geq_{ttt} E$. This stochastic inequality boils down to $S_F(p)\geq p, \forall p\in[0,1]$, where $S_F=T_F/\mu_F$ is the scaled TTT transform of $F$, while it is easy to see that the TTT transform of $E$ is just the identity function.
 
We now introduce a family of tests for $\mathcal{H}^{NBUE}_0:X/\mu_F\geq_{ttt} E$ versus $\mathcal{H}^{NBUE}_1:X/\mu_F\not\geq_{ttt} E$ based on a random sample of size $n$ from $F$. Goodness-of-fit tests for ageing properties of this type are especially relevant in life testing problems and in shape-constrained statistical inference. {There are many examples of tests of this type for the increasing hazard rate family, including the one of \cite{tenga1984}, also based on the TTT transform. Our goodness-of-fit approach to NBUE is complementary to that of the articles cited in Section~1, which consists in testing the null hypothesis of exponentiality versus strictly NBUE alternatives}. {However, in this setting, computing critical values requires either a least favorable distribution, as in \cite{landonbue}, or resorting to bootstrap, as in the present paper.}

Let now $S_{F_n}=T_{F_n}/\overline{X}$ be the empirical version of $S_F$. We can measure deviations from the null hypothesis by measuring the positive part of the function $\widehat{\delta}(p):=p-S_{F_n}(p)$, which is the empirical counterpart of $\delta(p)=p-S_F(p)$. Under $\mathcal{H}^{NBUE}_0$, $p-S_F(p)\leq0,\forall p\in[0,1]$. Therefore $\widehat{\delta}(p)\leq S_F(p)-S_{F_n}(p),\forall p\in[0,1]$. The following result, based on Theorem~\ref{process}, establishes the weak convergence of the process $s^F_n=\sqrt{n}(S_{F_n}-S_F)$.
\begin{corollary}\label{cor1}
Under Assumption~\ref{ass1}, for $n\to\infty$, $s_n^F\rightsquigarrow \mathcal{S}_F$ in $C[0,1]$,
	where $$\mathcal{S}_F(p)=\frac{\mathcal{T}_F(p)}{\mu_F}-\frac{T_F(p)\mathcal{T}_F(1)}{\mu_F^2}.$$
\end{corollary}

Since $\mathcal{S}_F$ is a centered Gaussian process, we can establish that, under $\mathcal{H}^{NBUE}_0$, $\sqrt{n}\Psi_r(\widehat{\delta})\leq \sqrt{n}\Psi_r(S_{F_n}-S_F)\rightsquigarrow\Psi_r(\mathcal{S}_F),$ where, for any $\alpha<1/2$, the $(1-\alpha)$ quantile of $\Psi_r(\mathcal{S}_F)$ is positive, finite, and unique. On the other hand, under $\mathcal{H}^{NBUE}_1$, the test statistic $\sqrt{n}\Psi_r(\widehat{\delta})$ diverges in probability. Now, we may approximate the distribution of $\sqrt{n}\Psi_r(S_{F_n}-S_F)$ using the same bootstrap procedure discussed earlier. In particular, let $S_{F^*_n}$ be the bootstrap version of $S_{F_n}$. The $p$-values may be approximated as
$p\approx 1/K \sum\nolimits_{k=1}^K \1(\sqrt{r_{n,m}}\Psi_r( (S_{F_n}-S_{F_n^*})_k)>\sqrt{r_{n,m}}\Psi_r( \widehat{\delta})),$
where $(S_{F_n}-S_{F_n^*})_k$ is the $k$-th realisation of the simulated process $S_{F_n}-S_{F_n^*}$. Accordingly, we may establish the following asymptotic properties.
\begin{proposition}
	\label{NBUE test}
	Under Assumption~\ref{ass1},
	\begin{enumerate}
		\item If $\mathcal{H}^{NBUE}_0$ is true, $\lim_{n\rightarrow\infty}P\{\text{reject }\mathcal{H}^{NBUE}_0\}\leq\alpha$; 
		\item If $\mathcal{H}^{NBUE}_0$ is false, $\lim_{n\rightarrow\infty}P\{\text{reject }\mathcal{H}^{NBUE}_0\}=1$.\end{enumerate}
\end{proposition}

\section{Tests for the excess wealth order}\label{sect ew}

Similarly to the case of the TTT order, $\mathcal{H}^{ew}_0$ is false if and only if the difference $\delta(p)=W_G(p)-W_F(p)$ is positive at some point $p\in[0,1]$. Letting ${W}_{F_n}(p)=\overline{X}_n-T_{F_{n}}$ and ${W}_{G_m}(p)=\overline{Y}_m-T_{G_{n}}$, we reject the null hypothesis if $\Psi_r(W_{G_m}-W_{F_n})$ is large enough. As $n$ and $m$ tend to infinity, $\widehat{\delta}:=W_{G_m}-W_{F_n}$ converges strongly and uniformly to $\delta$ in $[0,1]$.
The empirical excess wealth processes associated with $F$ and $G$ may be defined as $$w^F_n=\sqrt{n}(W_{F_n}-W_F), \qquad w^G_m=\sqrt{m}(W_{G_m}-W_G),$$
respectively. These processes yield $$\Delta_{n,m}:=\sqrt{r_{n,m}}(\widehat{\delta}-\delta)=\sqrt{\tfrac n{n+m}}w^G_m-\sqrt{\tfrac m{n+m}}w^F_n.$$
Weak convergence of $\Delta_{n,m}$ may be established in the following corollary, using the same arguments of Theorem~\ref{process} (the proof is omitted).
\begin{corollary}\label{processw}
 Under Assumption~\ref{ass1}, for the sampling schemes discussed in Section~\ref{sampling} we have$$ \sqrt{r_{n,m}}(\widehat{\delta}-\delta)\rightsquigarrow\widetilde{\mathcal{W}}=\sqrt{\lambda}\mathcal{W}_G-\sqrt{1-\lambda}\mathcal{W}_F\quad \text{ in }C[0,1],$$
	where
	$$\mathcal{W}_F(p)= -\frac{(1-p)\mathcal{B}_1(p)}{f\circ F^{-1}(p)}+\int_p^1 \frac{\mathcal{B}_1(y)}{f\circ F^{-1}(y)}dy,\qquad\mathcal{W}_G(p)= -\frac{(1-p)\mathcal{B}_2(p)}{g\circ G^{-1}(p)}+\int_p^1 \frac{\mathcal{B}_2(y)}{g\circ G^{-1}(y)}dy.$$
	
\end{corollary}

Using the same arguments of Section~\ref{consistency}, we may establish an equivalent result to Lemma~\ref{lemmatest}. Denoting the bootstrap estimators of $W_{F_n}$ and $W_{G_m}$ as $W_{F^*_n}$ and $W_{G^*_m}$, respectively, and letting $\widehat{\delta}^*=W_{G^*_n}-W_{F^*_n}$, the functional delta method for the bootstrap implies that $\sqrt{r_{n,m}}\Psi_r( \widehat{\delta}^*-\widehat{\delta})$ has the same limiting distribution of~$\sqrt{r_{n,m}}\Psi_r(\widetilde{\mathcal{W}}) $. This allows the approximation of $p$-values using bootstrap replications as in Section~\ref{boot}. Accordingly, by the same arguments used in the proof of Proposition~\ref{ptest}, we show consistency of the corresponding bootstrap-based tests.

\begin{proposition}
\label{ewtest}
Under Assumptions~\ref{ass1} and the sampling schemes in Section~3.1,
\begin{enumerate}
	\item If $\mathcal{H}^{ew}_0$ is true, $\lim_{n\rightarrow\infty}P\{\text{reject }\mathcal{H}^{ew}_0\}\leq\alpha$; 
	\item If $\mathcal{H}^{ew}_0$ is false, $\lim_{n\rightarrow\infty}P\{\text{reject }\mathcal{H}^{ew}_0\}=1$.\end{enumerate}
	\end{proposition}

\section{Simulations}\label{sect sim}
In this section, {we investigate the finite sample properties of our tests through some numerical analyses. In particular,} we focus on tests for $\mathcal{H}_0^{ttt}$ and $\mathcal{H}_0^{NBUE}$. {We do not focus on tests for the excess wealth order, because their behaviour is specular to that of our tests for the TTT order. We consider the Weibull and the Singh-Maddala distributions. Namely, we} denote with $W(a,b)$ the  Weibull distribution with shape parameter $a>0$, scale parameter $b>0$, and CDF $F_\text{W}(x;a,b)=1-\exp(-(x/b)^a)$. Moreover, we denote with $\text{SM}(a,b,c)$ the Singh-Maddala distribution with shape parameters $a>0,b>0$, scale parameter $c>0$, and CDF $F_{\text{SM}}(x;a,b,c)=1-\left((x/c)^b+1\right)^{-a}$. The TTT transform of $W(a,b)$ is 
$$T_{F_\text{W}}(p)=\frac{b}{a}\left(\Gamma \left(\frac{1}{a},0\right)- \Gamma \left(\frac{1}{a},-\log(1-p)\right)\right),$$
where $\Gamma$ is the gamma function, while that of $\text{SM}(a,b,c)$ exists just for $ab>1$ and has the following expression
$$ T_{F_\text{SM}}(p)=c \left((1-p)^{-1/a}-1\right)^{1/b} \, _2F_1\left(a,\frac{1}{b};1+\frac{1}{b};1-(1-x)^{-1/a}\right),$$
where $\, _2F_1$ is the hypergeometric function.

We simulate different scenarios from the distributions above, under the independent sampling framework, for $n=50,100,200,500,$ to investigate finite sample properties of the tests based on $\Psi_\infty$ and $\Psi_1$. For each scenario, we run 500 simulation experiments and sample 500 bootstrap replicates. The significance level of the test is set to $\alpha=0.1$. To simplify computations, we use the step function $\widetilde{T}_{F_n}(p)=\int_0^{F_n^{-1}(p)}(1-F_n(x))dx$ instead of $T_{F_n}$, since the difference among them is asymptotically negligible (recall that $T_{F_n}$ interpolates the jumps points of $\widetilde{T}_{F_n}$).
\subsection{Tests for $\mathcal{H}_0^{ttt}$}
We first let $X\sim \text{W}(a,1)$ and $Y\sim \text{W}(2,1)$. In this case, we have $X\geq_{ttt} Y$ when $a\geq2$, and the dominance becomes more evident for larger values of $a$. Both tests in these cases provide rejection rates that are always smaller than the nominal level $\alpha$ and get closer to 0 when the sample size grows. For this reason, we do not report these results here. Differently, letting $X\sim \text{W}(2,1)$ and $Y\sim \text{W}(a,1)$ we clearly have $X\not\geq_{ttt} Y$ for $a>1$. In this case, it is not difficult to detect deviations from the null hypothesis since $T_G$ is always above $T_F$. This is confirmed by the rejection rates in Figure~\ref{f1}. In this case, $\Psi_1$ tends to deliver larger power compared to $\Psi_\infty$. However, it may be harder to reject $\mathcal{H}_0^{ttt}$ when $T_G$ and $T_F$ cross. We consider the case in which $X\sim \text{W}(a,1)$ and $Y\sim \text{W}(1,1.2)$, for $a\geq1$, and also the reverse case, namely, $X\sim \text{W}(1,1.2)$ and $Y\sim \text{W}(a,1)$. The latter case is the most critical to detect, in particular for smaller values of $a$. Our results show that $\Psi_\infty$ often provides larger power compared to $\Psi_1$ in these scenarios. For $X\sim \text{W}(1,1.2)$ and $Y\sim \text{W}(1.25,1) $ both tests (especially $\Psi_1$) struggle to reject the null hypothesis, at least for the sample sizes considered (we obtain larger power increasing $n$ to 2000). However, even in this critical case, the $p$-values tend to decrease as the sample size grows, coherently with the consistency property in Proposition \ref{ptest}. Overall, our results suggest that $\Psi_\infty$ may be more reliable than $\Psi_1$, especially when the TTT transforms cross.

\begin{figure}
	\centering
	\includegraphics[scale=0.7]{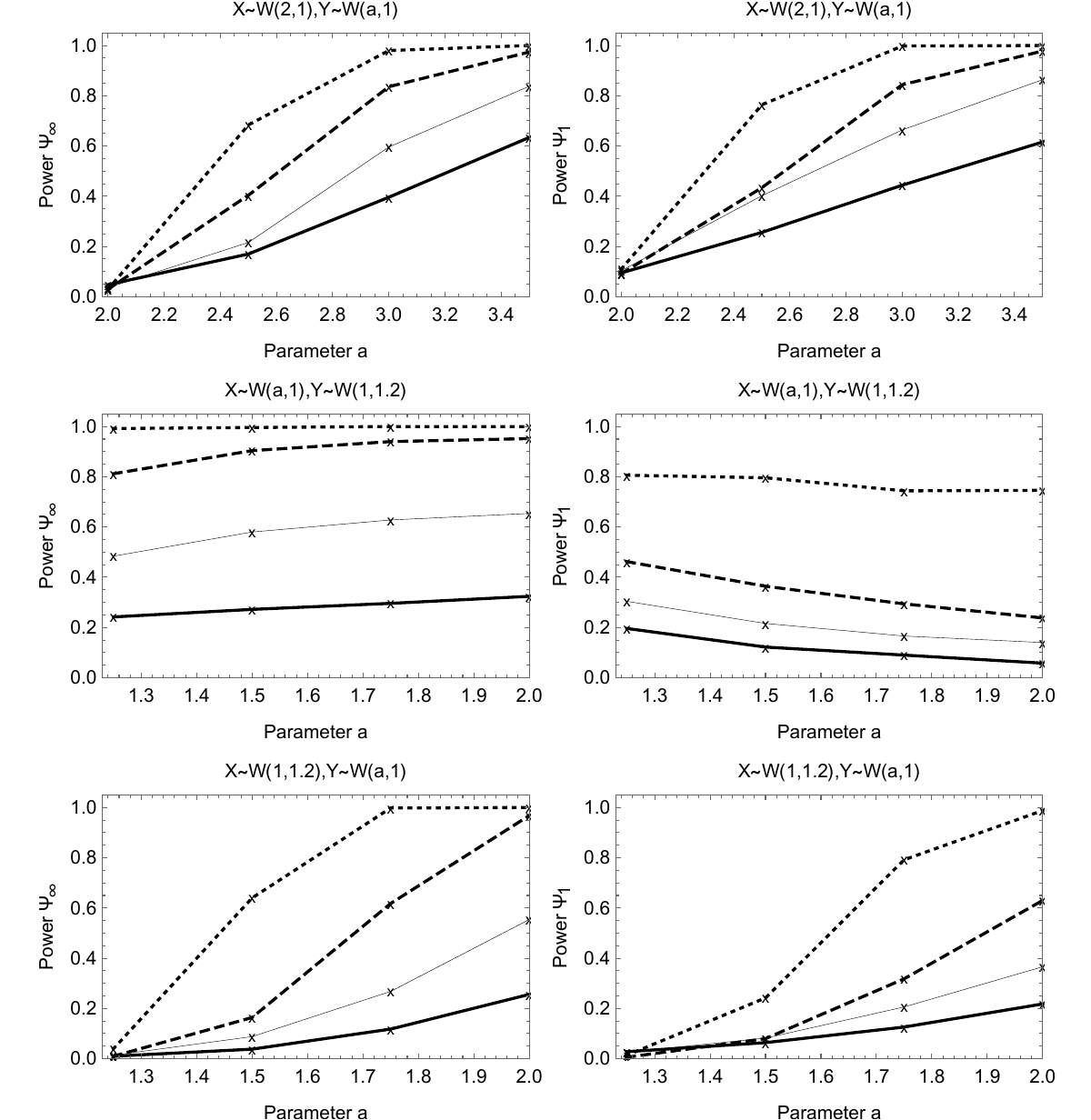}
	\caption{Tests for $\mathcal{H}_0^{ttt}$. Rejection rates for sample sizes $n=50$ (solid), $n=100$ (solid, thin), $n=200$ (dashed), $n=500$ (dotted).\label{f1}}
\end{figure}
\subsection{Tests for $\mathcal{H}_0^{NBUE}$}
First, recall that ageing properties are scale-independent, so there is no loss of generality in setting scale parameters equal to 1. The Weibull distribution $\text{W}(a,b)$ is NBUE for $a\geq1$, while, for $a<1$, it belongs to the \textit{new worse than used in expectation} (NWUE) class, which is the family of distributions characterised by the reverse stochastic inequality, $E\geq_{ttt}X/\mu_F$. Letting $X\sim \text{W}(a,1),$ where $a$ varies within $[0.8,1.1]$, we may observe the performance of the tests under the null and the alternative hypotheses. The results, reported in Figure~\ref{f2}, show that $\Psi_1$ provides larger power under the alternative compared to $\Psi_\infty$. The latter seems to be more conservative, as it appeares also by observing the behaviour of the tests under the null hypothesis. For both tests, the rejection rates tend to 1 under the alternative and tend to 0 under the null, confirming the asymptotic consistency properties. Moreover, the empirical power is always above the nominal level 0.1 when $a<1$, showing an unbiased behaviour against NBUE alternatives.


We also generate samples from the $\text{SM}(a,b,c)$ distribution, which is neither NBUE nor NWUE for $a\in[1,2]$ and $b=1.5$. Actually, in this case, the scaled TTT transform $S_{F_{\text{SM}}}$ crosses the identity function, which is the TTT transform of the unit exponential, from above. In particular, it is more difficult to detect non-dominance for larger values of $a$. The simulation results show that, in this more critical case, $\Psi_\infty$ yields larger rejection rates compared to $\Psi_1$, which seems to be less sensitive in detecting violations of the NBUE property especially for $a=2$. 

\begin{figure}
	\centering
	\includegraphics[scale=0.7]{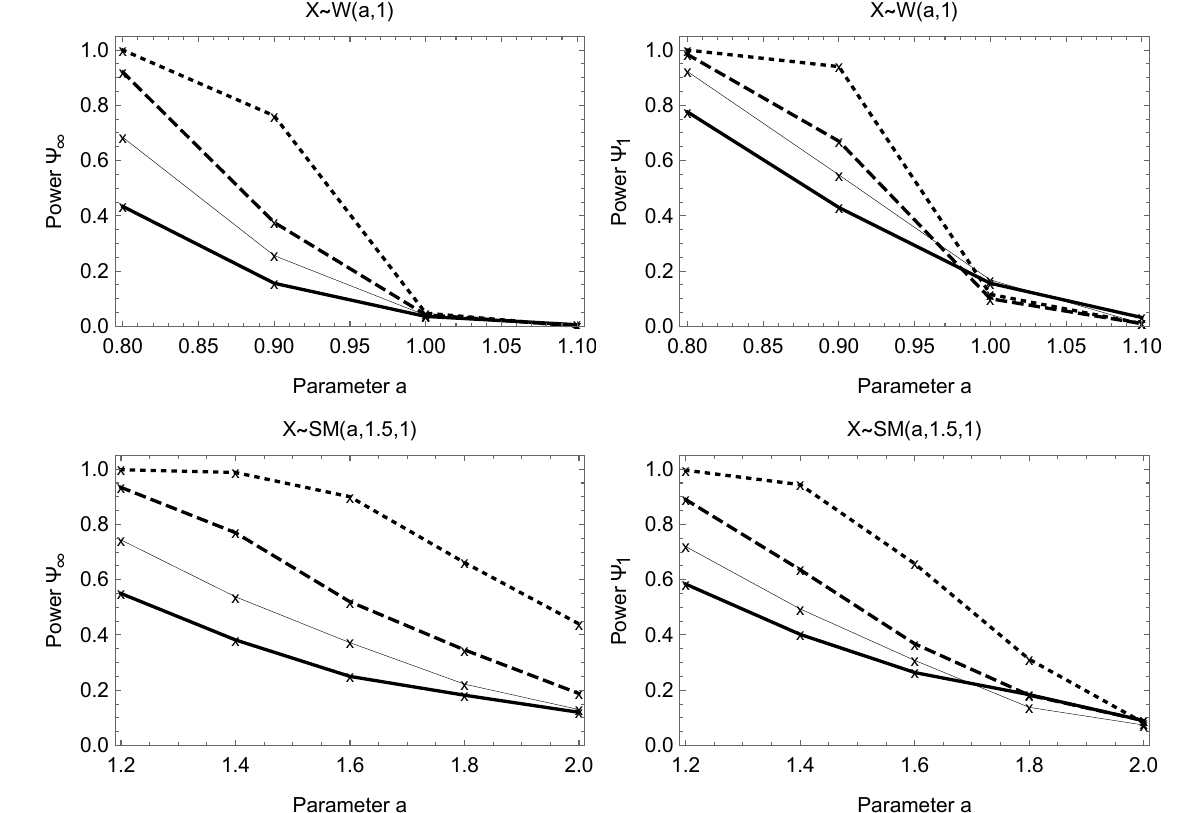}
	\caption{Tests for $\mathcal{H}_0^{NBUE}$. Rejection rates for sample sizes $n=50$ (solid), $n=100$ (solid, thin), $n=200$ (dashed), $n=500$ (dotted).\label{f2}}
\end{figure}
 \section{An example of application to survival data}\label{sect ex}

{As an illustrative example, we apply our tests on the dataset in \cite{hoel}, which reports the survival times of some mice which had received a radiation dose at an age of 5-6 weeks. The mice affected by the two main diseases, namely, thymic lymphoma (51 mice) and reticulum cell sarcoma (53 mice), were separated from the others (77 mice).}
 
{According to biologists, these two diseases are lethal and they are independent of one another and of other causes of death. One may be interested in testing whether the lifetime of the mice affected by these major diseases is stochastically smaller or larger (in the usual sense) than that of the other mice. In particular, tests for the usual stochastic order \citep{bd} suggest that thymic lymphoma shortens the lifespan significantly, compared to reticulum cell sarcoma and to the other causes of death. On the other hand, these tests also reveal that the lifetime of the mice affected by reticulum cell sarcoma, denoted with $X$, is not stochastically smaller than that of the mice which died by other minor causes, denoted with $Y$ (in this case, the  hypothesis of dominance is rejected with a $p$-value close to 0). In particular, for small values of $p$, the $p$-quantiles of $X$ are apparently larger than those of $Y$ (see the left panel of Figure~\ref{f4}), meaning that some other causes lead to death more rapidly than reticulum cell sarcoma. Moreover, sample means do not differ too much, but the sample mean of $X$ (i.e., 635) is slightly larger compared to that of $Y$ (i.e., 565), owing to one outlier in $X$ (i.e., 986). However, the non-affected mice seem to live longer at larger quantile values.}
	
	{It is then useful to understand whether weaker orders hold. It appears from Figure~\ref{f4} that we might have $X\leq_{ew}Y$ and $X\geq_{ttt}Y$. By applying our test for the excess wealth order, we find a $p$-value around 0.5 for both $\Psi_1$ and $\Psi_\infty$. The tests of \cite{belzunce2001}, at least for some values of their parameter $\alpha,$ confirm that the excess wealth order may hold; moreover, in this case, the hypothesis $X=_dY$ is easily rejected by the Kolmogorov-Smirnov test. Hence, the data suggest that $X\leq_{ew}Y$ (strictly). In other words, for every $p\in[0,1]$, the mice that survive beyond the $p$-quantile seem to have a longer expected residual lifetime when they are not affected by this disease. Similarly, by applying our test for the TTT order, we obtain even larger $p$-values, indicating also that $X\geq_{ttt}Y,$ as confirmed by any of the tests in \cite{belzunce2005}. So, vice-versa, the expected lifetime of the mice that do not survive more than the $p$-quantile is smaller for the other (and possibly, more rapid) causes of death.}

{This example illustrates how the tests proposed in the present paper can assist in the analysis of survival data, in particular when stronger stochastic orders are not confirmed by tests, but the data suggest that the weaker orders considered in this paper may hold. Our tests can be used as an alternative, or complement, to the those proposed by \cite{belzunce2001,belzunce2005}, which in fact are not tailored to work under $\mathcal{H}_1^{ttt}$ or $\mathcal{H}_1^{ew}$. This is especially useful when the hypothesis of equal distributions is rejected by a goodness-of-fit test, as in the example considered in this section.}

\begin{figure}
	\centering
	\includegraphics[width=14cm,height=9cm]{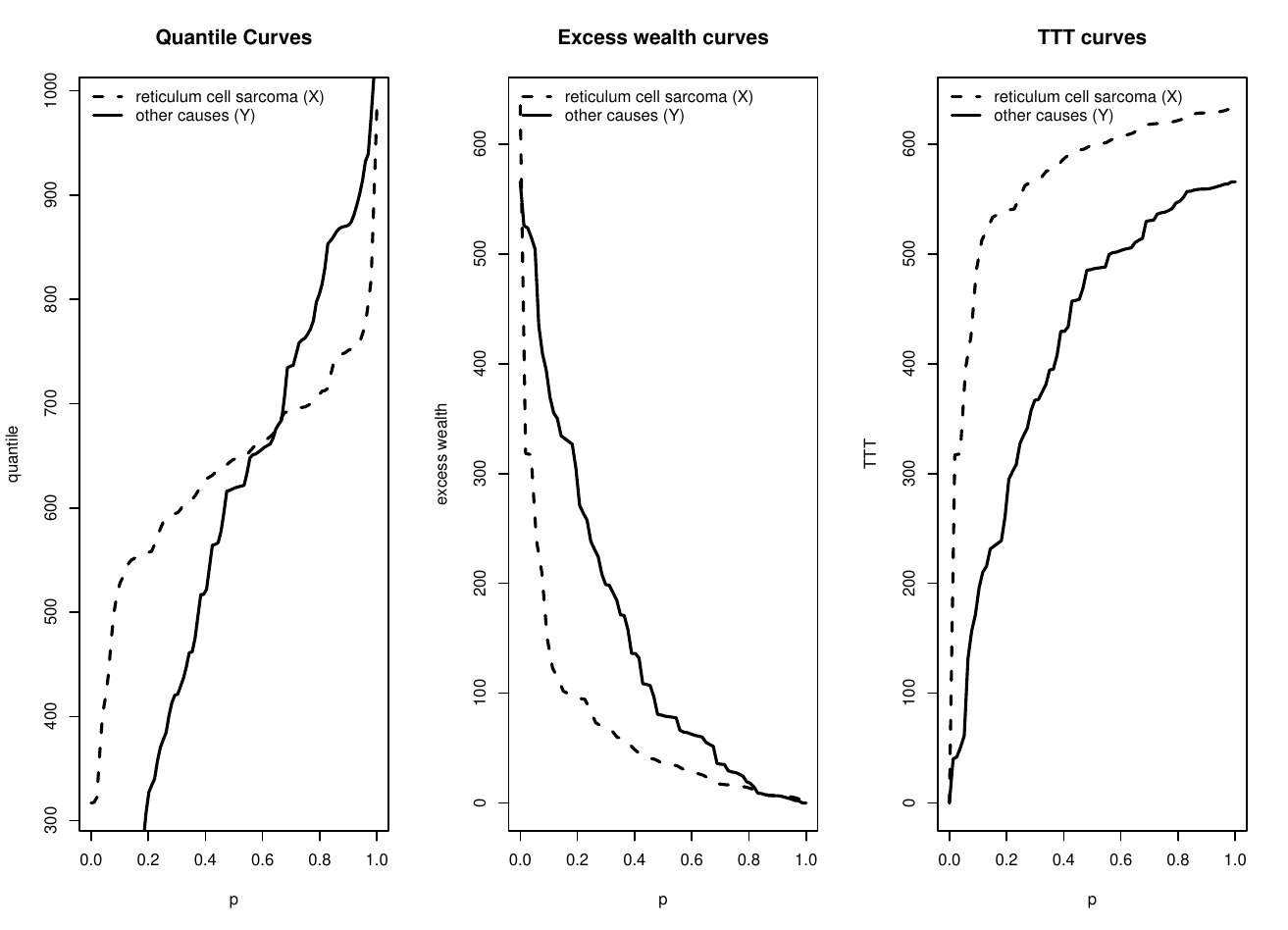}
	\caption{Empirical quantile curves and excess wealth transforms of the data in Section 7. The lifetime of the mice affected by reticulum cell sarcoma ($X$) appears to be smaller in the excess wealth order, and larger in the TTT order.\label{f4}}
\end{figure}

	\section*{Appendix: Proofs}
\label{app_proofs}

\begin{proof}[Proof of Theorem~\ref{process}.]
	Let $\mathbb{L}$ be the space of maps $z:[0,\infty)\rightarrow\mathbb{R}$ with $\lim_{x\rightarrow -\infty}z(x)=0$ and $\lim_{x\rightarrow\infty}z(x)=1$, and the norm $||z||_\mathbb{L}=\max\{||z||_\infty,||1-z||_1\}$. Moreover, let $\mathbb{B}$ be the Banach space of measurable functions $z$ from $(0, 1)$ to $\R$ equipped with the norm $||z||_{\mathbb{B}}=\int_0^1|z(p)|dp$. As shown by \cite{kaji}, under Assumption~\ref{ass1}, the map $\theta(F)=F^{-1}$, from CDFs to quantile functions, is Hadamard differentiable at $F$, tangentially to the set $\mathbb{L}_0$ of continuous functions in $\mathbb{L}$, with derivative map $\theta'_F:\mathbb{L}_0\to \mathbb{B}$ given by $\theta_F'(z)=-(z\circ F^{-1})(F^{-1})'. $
	Now, we need to establish the Hadamard differentiability of the map $\Phi:F^{-1}\rightarrow T_F$ as a map from $\mathbb{B}$ to $C[0,1]$. {By the change of variable $y=F(x)$, and then integrating by parts, we can write $$T_F(p)=\int_0^{F^{-1}(p)}(1-F(x))dx=\int_0^{p}(1-y)dF^{-1}(y)=F^{-1}(p)(1-p)+\int_0^pF^{-1}(y)dy.$$
	By linearity of the operator $\phi(F^{-1})$, the Hadamard derivative of $\phi$ at $F^{-1}$ in the direction $\alpha$ is still $\phi$, namely $\phi'_{F^{-1}}(\alpha)=\phi(\alpha)$. That is, letting $F^{-1}_t=F^{-1}+t\alpha_t$, where $||\alpha_t-\alpha||_1\to0$ and $t\rightarrow0$,
	$$\bigg\lvert\bigg\lvert\frac{\phi(F_t^{-1})-\phi(F^{-1})}t-\phi'_{F^{-1}}(\alpha) \bigg\lvert\bigg\lvert_1\leq\int_0^1\big\lvert(1-p)(\alpha_t(p)-\alpha(p))\big\lvert dp+\int_0^1\int_0^p\big\lvert\alpha_t(y)-\alpha(y)\big\lvert dy dp\to0.$$
	By the chain rule for Hadamard differentiability \citep[Lemma 3.9.3]{vw}, we have that the map $\phi\circ\theta:F\rightarrow T_F$ is Hadamard differentiable at $F$, tangentially to $\mathbb{L}_0$, with derivative
	$$(\phi\circ\theta)'_F(z(p))=\phi'_{\theta(F)}\circ \theta'_F(z(p))=(1-p)\theta_F(z(p))+\int_0^p\theta_F(z(y))dy
	=-(1-p)(z\circ F^{-1}(p))(F^{-1}(p))'-\int_0^p (z\circ F^{-1}(y))dF^{-1}(y).$$}
	Now, observe that $$\begin{pmatrix}\sqrt{n}(F_n-F) \\ \sqrt{m}(G_m-G)\end{pmatrix}\rightsquigarrow \begin{pmatrix}\mathcal{B}_1\circ F \\ \mathcal{B}_2\circ G\end{pmatrix}\text{  in }\mathbb{L}\times\mathbb{L},$$
	as shown in Lemma 5.1 of \cite{sunbeare}. Then, the functional delta method \citep[Theorem 3.9.4]{vw} implies the joint weak convergence 
	\begin{multline}
		\label{lor process}
		\begin{pmatrix}\sqrt{n}(T_{F_n}(p)-T_F(p)) \\ \sqrt{m}T_{G_m}(p)-T_G(p))\end{pmatrix}=\begin{pmatrix}\sqrt{n}(\phi\circ\theta({F_n})(p)-\phi\circ\theta({F})(p)) \\ \sqrt{m}(\phi\circ\theta({G_m})(p)-\phi\circ\theta({G})(p))\end{pmatrix} \rightsquigarrow\begin{pmatrix} (\phi\circ\theta)'_F(\mathcal{B}_1\circ F)(p) \\ (\phi\circ\theta)'_G(\mathcal{B}_2\circ G)(p)\end{pmatrix}\\=
		\begin{pmatrix}-(1-p)\mathcal{B}_1(p)(F^{-1})'(p)-\int_0^p \mathcal{B}_1(y)dF^{-1}(y) \\-(1-p)\mathcal{B}_2(p)(G^{-1})'(p)-\int_0^p \mathcal{B}_2(y)dG^{-1}(y)\end{pmatrix} =\begin{pmatrix}-\mathcal{T}_F(p)\\ -\mathcal{T}_G(p)\end{pmatrix}=_d \begin{pmatrix}\mathcal{T}_F(p)\\ \mathcal{T}_G(p)\end{pmatrix}\text{  in }C[0,1]\times C[0,1].\end{multline}
	The last relation ($=_d$) is because $\mathcal{T}_F$ and $\mathcal{T}_G$ are centered Gaussian processes.
	Then the results follow from the continuous mapping theorem, as in the proof of Lemma 2.1 of \cite{sunbeare}.
	\end{proof}
	
	\begin{proof}[Proof of Proposition~\ref{ptest}].
		As proved in Lemma 5.2 of \cite{sunbeare},
		$$\begin{pmatrix} \sqrt{n}({F^*_n}-{F_n} )\\ \sqrt{m}({G^*_m}-{G_m}) \end{pmatrix}\vc{\rightsquigarrow}{as*}{M} \begin{pmatrix} \mathcal{B}_1\circ F \\ \mathcal{B}_2\circ G \end{pmatrix} \text{  in }\mathbb{L}\times \mathbb{L},$$
		where $\vc{\rightsquigarrow}{as*}{M}$ denotes weak convergence conditional on the data almost surely, see \citet[p.20]{kosorok}.
		The proof of Theorem~\ref{process} establishes the Hadamard differentiability of the map $\phi \circ \theta:H\rightarrow T_H$, so that the functional delta method for the bootstrap implies
		$$\sqrt{r_{n,m}}\begin{pmatrix} T_{F^*_n}-T_{F_n} \\ T_{G^*_n}-T_{G_m} \end{pmatrix}\vc{\rightsquigarrow}{P}{M} \begin{pmatrix} \lambda \mathcal{T}_F \\ (1-\lambda) \mathcal{T}_G \end{pmatrix} \text{ in }C[0,1]\times C[0,1],$$
		where $\vc{\rightsquigarrow}{P}{M}$ denotes weak convergence conditional on the data in probability \citep[see][p.20]{kosorok}.
		This entails that $\Psi_r(\sqrt{r_{n,m}}( \widehat{\delta}_k^*-\widehat{\delta}))\vc{\rightsquigarrow}{P}{M} \Psi_r(\widetilde{\mathcal{T}})$ by the continuous mapping theorem. The test rejects the null hypothesis if the test statistic exceeds the bootstrap threshold value $c^*_n(\alpha)=\inf\{y: P(\sqrt{r_{n,m}} \ \Psi_r( \widehat{\delta}^*-\widehat{\delta})>y|{\mathcal{X},\mathcal{Y}})\leq \alpha\}$. However, the weak convergence result implies $c_n^*(\alpha)\rightarrow c(\alpha)=\inf\{y: P(\Psi_r(\widetilde{\mathcal{T}})>y)\leq \alpha\}$ in probability, so Lemma~\ref{lemmatest} yields the result.
	\end{proof}
	
	\begin{proof}[Proof of Corollary~\ref{cor1}]
		Using the same notations for $\theta$ and $\Psi$ as in the proof of Theorem~\ref{process}, we now focus on the map $\eta\circ\Psi\circ\theta:F\to S_F$, where $\eta:C[0,1]\to C[0,1]$ is defined as $\eta(T_F)=T_F(\cdot)/T_F(1)$. The Hadamard derivative of $\eta$ at $T_F$ in the direction $u\in C[0,1]$ is given by
		$$ \eta'_{T_F}(u)=\frac{u(\cdot)}{T_F(1)}-\frac{u(1)T_F(\cdot)}{T_F(1)^2},$$
		see the proof of Lemma 2.1 in \cite{sunbeare}. Therefore, the result follows from the chain rule and the functional delta method.
	\end{proof}  
	
\section*{Acknowledgements}
We are grateful to the anonymous referees for their useful suggestions and constructive comments.

\bibliographystyle{elsarticle-harv}
\bibliography{biblio}





\end{document}